\documentclass{amsart}
\usepackage{hyperref, fullpage}

\newtheorem{theorem}{Theorem}[section]
\newtheorem{lemma}[theorem]{Lemma}

\theoremstyle{definition}
\newtheorem{definition}[theorem]{Definition}
\newtheorem{remark}[theorem]{Remark}

\numberwithin{equation}{theorem}

\newcommand{\FF}{\mathbb{F}}
\newcommand{\QQ}{\mathbb{Q}}

\newcommand{\ZZ}{\mathbb{Z}}

\newcommand{\frakp}{\mathfrak{p}}
\newcommand{\PP}{\mathbf{P}}

\newcommand{\ptcount}[2]{(\##1(\FF_{2^i}))_{i=1}^{#2}}
\newcommand{\tracecount}[2]{(T_{#1,2^i})_{i=1}^{#2}}

\DeclareMathOperator{\Gal}{Gal}
\DeclareMathOperator{\PGL}{PGL}
\DeclareMathOperator{\PSL}{PSL}
\DeclareMathOperator{\Res}{Res}

\newcommand{\avlink}[1]{\href{http://www.lmfdb.org/Variety/Abelian/Fq/#1}{\textsf{#1}}}

\newcommand{\Magma}{\textsc{Magma}}
\newcommand{\SageMath}{\textsc{SageMath}}

\newcommand{\arXiv}[3]{\href{https://arxiv.org/abs/#1}{arXiv:#1v#2} (#3)}

\begin{document}

\thanks{Thanks to Everett Howe for helpful discussions, particularly about Lemma~\ref{lem:degree 7 cyclic} and Lemma~\ref{lem:genus 6 cover}.
Kedlaya was supported by NSF (grants DMS-1802161, DMS-2053473) and UC San Diego (Warschawski Professorship).}
\date{May 29, 2024}

\title{The relative class number one problem for function fields, II}
\author{Kiran S. Kedlaya}

\begin{abstract}
We establish that any finite extension of function fields of genus greater than 1 whose relative class group is trivial is Galois and cyclic. This depends on a result from a preceding paper which establishes a finite list of possible Weil polynomials for both fields.
Given this list, we analyze most cases by computing options for the splittings of low-degree places in the extension, then consider the effect of these options on the Weil polynomials of certain isogeny factors of the Jacobian of the Galois closure.
In one case, we use instead an analysis based on principal polarizations, modeled on an argument of Howe.
\end{abstract}

\maketitle

\section{Introduction}

This paper continues the work done in \cite{part1} on the \emph{relative class number one problem} for function fields, building upon work of Leitzel--Madan \cite{leitzel-madan} and Leitzel--Madan--Queen \cite{leitzel-madan-queen}.
That is, we seek to identify finite extensions $F'/F$ of function fields of curves over finite fields for which the two class numbers are equal. In this paper, we establish the following result;
the argument relies heavily on results from \cite{part1} as described below.

\begin{theorem} \label{T:cyclic}
Let $F'/F$ be a finite extension of function fields of genus greater than $1$ with the same constant field and class number. Then $F'/F$ is Galois and cyclic.
\end{theorem}
Before continuing, we recall some notation from \cite{part1}. Given a finite extension $F'/F$ of function fields, we write
$C, C'$ for the curves corresponding to $F,F'$;
$q_F, q_{F'}$ for the orders of the base fields of $C,C'$;
$g_F, g_{F'}$ for the genera of $C, C'$;
and $h_F, h_{F'}$ for the class numbers of $F, F'$. 
In this paper we only consider the case where $q_F = q_{F'}$, in which case we say $F'/F$ is a \emph{purely geometric extension} (the other extreme is when $g_{F'} = g_F$, in which case we say $F'/F$ is a \emph{constant extension}). We recall here \cite[Theorem~1.3(a)]{part1}.

\begin{theorem} \label{T:purely geometric bounds}
Let $F'/F$ be a finite purely geometric extension of function fields with $g_{F'} > g_F > 1$ and $h_{F'} = h_F$.
\begin{enumerate}
\item[(a)]
If $q_F>  2$, then the tuple $(q_F, d, g_F, g_{F'}, F)$ is listed in \cite[Table~4]{part1}.
Moreover, $F'/F$ is cyclic.
\item[(b)]
If $q_F = 2$, then the triple $(d, g_F, g_{F'})$ is listed in Table~\ref{table:genus triples}.
Moreover, the pair of Weil polynomials of $C$ and $C'$ belong to an explicit list of $208$ triples (see below).
\end{enumerate}
\end{theorem}

For brevity we do not include the complete list of Weil polynomials in Theorem~\ref{T:purely geometric bounds}(b); instead, we catalog in \S\ref{sec:numerical data} all of the specific features of this list that we need here (see especially Table~\ref{table:weil polynomials}). The full table can be found in an Excel spreadsheet in the repository \cite{repo}.

\begin{table}[ht]
\begin{tabular}{c|ccc|cc|cc|c|c|c||cc|c|c||cc|c||c||c||c}
$d$ & \multicolumn{10}{c||}{$2$}  & \multicolumn{4}{c||}{$3$} & \multicolumn{3}{c||}{$4$}& $5$ & $6$ & $7$ \\
\hline
$g_F$ & \multicolumn{3}{c|}{$2$} & \multicolumn{2}{c|}{$3$} & \multicolumn{2}{c|}{$4$} & $5$ & $6$ & $7$ & \multicolumn{2}{c|}{$2$} & $3$ & $4$ & \multicolumn{2}{c|}{$2$} & $3$ & $2$ & $2$ & $2$ \\
\hline
$g_{F'}$ & $3$ & $4$ & $5$ & $5$ & $6$ & $7$ & $8$ & $9$ & $11$ & $13$ & $4$ & $6$ & $7$ & $10$ & $5$ & $6$ & $9$ & $6$ & $7$ & $8$ 
\end{tabular}
\caption{Options for the triple $(d, g_F, g_{F'})$, as established in \cite{part1}.}
\label{table:genus triples}
\end{table}

To establish Theorem~\ref{T:cyclic}, it thus suffices to check the claim for $F$ having one of the Weil polynomials allowed by Theorem~\ref{T:purely geometric bounds}(b).
We organize this according to the value of $d$.
(For a uniform description of the paradigm applicable to $d=3,4,5,6$, see Remark~\ref{R:basic paradigm}.)
\begin{itemize}
\item
For $d=2$ there is nothing to check, as $F'/F$ is automatically Galois and cyclic.
\item
For $d=3,4$, we analyze the quadratic resolvent of $F'/F$ by using the point counts of $C$ and $C'$ to constrain the splitting of places from $C$ to $C'$ (Lemma~\ref{L:degree 3 cyclic}, Lemma~\ref{lem:degree 4}). This can be seen as a natural continuation of the arguments
of \cite[\S 8]{part1}.
\item
For $d=5$, we make a similar argument, but now accounting more fully for the representation theory of $A_5$, its effect on the isogeny splitting of the Jacobian of the Galois closure (as in \cite{paulhus}), and the point counts on additional quotients of the Galois closure of degree 6 and 10 over $C$ (Lemma~\ref{L:degree 5 cyclic}).
\item
For $d=6$, we make a similar argument, but using the quotient of the Galois closure of degree
6 over $C$ arising from the outer automorphism of $S_6$ (Lemma~\ref{lem:degree 6 cyclic}).
\item
For $d=7$, the similar argument becomes so complicated that we did not attempt to execute it.
Fortunately, we only have to analyze one pair of Weil polynomials, which is known to occur for a cyclic cover.
Even more fortunately, this example is amenable to an approach of Howe \cite{howe2021}:
we analyze the possible principal polarizations of abelian varieties in the isogeny class of the Jacobian of $C'$ to show that the cover is forced to be cyclic (Lemma~\ref{lem:degree 7 cyclic}).
\end{itemize}
In passing, we note some resemblance between these methods and the work of Rigato \cite{rigato} classifying curves of low genus over $\FF_2$ with the maximum number of $\FF_2$-rational points.
A pithy slogan for the overall strategy might be ``the most radical [extreme] covers are radical [cyclic]'': noncyclic covers have to meet the Weil polynomial constraint for multiple isogeny factors of the Galois closure, so their point counts have fewer degrees of freedom along which to vary to achieve extreme values.

As a side effect of the analysis for $d=7$, we also establish another missing assertion from \cite{part1}: there is exactly one curve of genus 6 over $\FF_2$ which occurs as a cover of a curve of genus 1 with relative class number 1 (Lemma~\ref{lem:genus 6 cover}). This confirms the corresponding entry in
\cite[Table~3]{part1}.

In light of Theorem~\ref{T:cyclic} and the results of \cite{part1}, the relative class number one problem now reduces to the case of purely geometric quadratic extensions of function fields over $\FF_2$. This requires identifying curves of genus 6 or 7 over $\FF_2$ with some specified Weil polynomials; we leave this task to a separate paper \cite{part3}.

Although it costs us some shortcuts to do so (Remark~\ref{R:castelnuovo-severi}),
it is possible to limit the use of the hypothesis that $h_{F'} = h_F$ to the list of pairs of Weil polynomials provided by Theorem~\ref{T:purely geometric bounds}.
Consequently, our methods can in principle be adapted to the relative class number $m$ problem for $m>1$, though in that case one expects to find some cases where a noncyclic extension does exist, and therefore additional arguments are needed to find \emph{all} extensions with the corresponding Weil polynomials. For $d=3,4,5$ it may be feasible to use Bhargava's orbit parametrizations \cite{bhargava3, bhargava4, bhargava5} for this purpose; for $d=3$, an analogous computation for number fields has been implemented by Belabas \cite{belabas}.

As in \cite{part1}, the arguments depend on a number of computations in \SageMath{} and \Magma{}; we have documented these in some Jupyter notebooks associated to this paper, which are available from a GitHub repository \cite{repo}. The computations can be reproduced in less than a minute on a single CPU (Intel i5-1135G7@2.40GHz). In a few places, we adopt the convention of \cite{part1} of referring to isogeny classes of abelian varieties over $\FF_2$ via their labels in LMFDB \cite{lmfdb}, formatted as links to the corresponding webpages.

\section{Numerical data}
\label{sec:numerical data}

\begin{definition}
For the remainder of the paper, let $F'/F$ be an extension of function fields of degree $d>1$ 
with $q_F = q_{F'} = 2$,
$g_{F'} > g_F > 1$, and $h_{F'/F} = 1$. 
We typically write $g,g'$ instead of $g_F, g_{F'}$.
Let $C' \to C$ be the corresponding cover of curves over $\FF_2$.
For $i$ a positive integer, let $a_i(F)$ and $a_i(F')$ denote the number of degree-$i$ places of $F$ and $F'$, respectively.
Recall that by Riemann--Hurwitz, $F'/F$ is everywhere unramified if and only if
\begin{equation} \label{eq:everywhere unramified}
g' = g + d(g-1).
\end{equation}
\end{definition}

The main purpose of this section is to record some partial information from \cite[Theorem~1.3(b)]{part1} about the possible Weil polynomials of the curves $C$ and $C'$. 
To begin with, we read off the following statement.
\begin{lemma} \label{lem:cycle1 prelim}
We have 
\begin{equation} \label{eq:cycle1 prelim}
a_1(F) > \sum_{i=1}^{\lfloor (d-1)/2 \rfloor} a_i(F')
+ \begin{cases} \frac{1}{2} a_{d/2}(F') & \mbox{ if \eqref{eq:everywhere unramified} holds,} \\
a_{d/2}(F') & \mbox{ otherwise,} \end{cases}
\end{equation}
except in one case where
\begin{equation} \label{eq:cycle1 prelim2}
d = 5, \, g = 2, \, g' = 6, \, a_1(F) = 5, \, a_1(F') = 0, \, a_2(F') = 5, \, a_3(F') = 0.
\end{equation}
\end{lemma}

While we can mostly ignore the case $d=2$ because any quadratic extension is automatically cyclic, we will need a few features of that case in order to analyze larger $d$.
\begin{remark} \label{R:genus 2 data}
In case $d=2$, the following statements hold.
\begin{itemize}
\item
For $(g,g') = (2,3)$, we have 
\[
(\#C(\FF_2), \#C(\FF_4); \#C'(\FF_2)) \in \{ (2,8; 0), (4,8; 2)\};
\]
moreover, by \cite[Theorem~1.3(c)]{part1}, both cases are possible.
\item
For $(g,g') = (3,5)$, we have $\#C'(\FF_2) \leq 2$. Moreover, if $\#C(\FF_2) = 4$ then $\#C(\FF_4) = 8$.
\item
For $(g,g') = (3,6)$, we have $\#C(\FF_2) \geq 4$.
\item
For $(g,g') = (4,7)$, we have $\#C(\FF_2) \geq 3$.
\item
For $(g,g') = (5,9)$, we have $\#C(\FF_2) \geq 2$. Moreover, if $\#C(\FF_2) = 2$ then $\#C'(\FF_4) = 4$.
\end{itemize}
\end{remark}

\begin{remark} \label{R:genus 2 data2}
Remark~\ref{R:genus 2 data} has the following specific consequence which will be of special interest. According to LMFDB (see \avlink{2.2.ab\_c}), there is exactly one curve $C_1$ of genus 2 over $\FF_2$
with $\ptcount{C_1}{2} = (2,8)$.
If $C' \to C_1$ is a degree-2 \'etale covering, then either this covering or its relative quadratic twist has relative class number 1.

By the same token, there is exactly one curve $C_2$ of genus 2 over $\FF_2$
with $\ptcount{C_2}{2} = (4,8)$, namely the quadratic twist of $C_1$.
If $C' \to C_2$ is a degree-2 \'etale covering, then again either this covering or its relative quadratic twist has relative class number 1.
\end{remark}

Table~\ref{table:weil polynomials} summarizes the possible point counts for $C$ and $C'$ when $d > 2$; we include only information needed in our proofs, leading to some gaps in the table.

\begin{table}[ht]
\Small
\begin{tabular}{c|c|c|c|c|c|c|c}
& $d$ & $g$ & $g'$ & $\#J(C)(\FF_2)$ & $\#J(C)(\FF_4)$ & $\#C(\FF_{2^i})$ & $\#C'(\FF_{2^i})$ \\
\hline
(1) & $3$ & $2$ & $4$ & 3/9/7/13 & & & $0/1$ \\
(2) & $3$ & $2$ & $4$ & & & $3$ & $0$ \\
\hline
(3) & $3$ & $2$ & $6$ & & & $3$/$4$/$5$ & $0, 2$ \\
\hline
(4) & $3$ & $3$ & $7$ & 23/27 & & $4$ & \\
(5) & $3$ & $3$ & $7$ & & & $2,8,14$ & $0,0$  \\
(6) & $3$ & $3$ & $7$ & & & $3,7$ & $0,0$ \\
(7) & $3$ & $3$ & $7$ & & & $4,6/8$ & $0,0$ \\
(8) & $3$ & $3$ & $7$ & & & $4,12$ & $0,6$ \\
(9) & $3$ & $3$ & $7$ & & & $5$ & $1,1$ \\
(10) & $3$ & $3$ & $7$ & & & $5,9$ & $1,3$ \\
(11) & $3$ & $3$ & $7$ & & & $6$ & $2,2,14$ \\
\hline
(12) & $3$ & $4$ & $10$ & & & $6,6$ & $0,0$ \\
(13) & $3$ & $4$ & $10$ & & & $7$/$8$ & $0$ \\
\hline
\hline
(14) & $4$ & $2$ & $5$ & & & & $1,*,1$ \\
(15) & $4$ & $2$ & $5$ & $4$ & $40$ & $2,8$ & $0,0$ \\
(16) & $4$ & $2$ & $5$ & $8$ & $16$ & $4,4$ & $0,4$ \\
(17) & $4$ & $2$ & $5$ & $10$ & $40$ & $4,8$ & $0,8$ \\
\hline
(18) & $4$ & $2$ & $6$ & & & $5$ & $1,1, *, 17$ \\
\hline
(19) & $4$ & $3$ & $9$ & $36$ & & $5, 9$ & $0, 0, *, 28$ \\
(20) & $4$ & $3$ & $9$ & $50$ & $200$ & $6, 8$ & $0, 0$ \\
\hline
\hline
(21) & $5$ & $2$ & $6$ & &   & $4$ & $0,6,0,18,0$ \\
(22) & $5$ & $2$ & $6$ & $11$ & & $4,10,7$ & $0,2,15$ \\
(23) & $5$ & $2$ & $6$ &  $19$ & & $6,6$ & $1,3,7$ \\
(24) & $5$ & $2$ & $6$ &  $15$ & & $5,9,5$ & $0,6,3$ \\
(25) & $5$ & $2$ & $6$ & $5$ & & $3,5,9,33,33$ & $0,0,0,20,15$ \\
(26) & $5$ & $2$ & $6$ &    & & $4,8,10,24,14$ & $0,0,15,20,20$ \\
(27) & $5$ & $2$ & $6$ &  $15$ & & $5,9,5,17,25$ & $0,10,0,10,25$ \\
(28) & $5$ & $2$ & $6$ &    & & $3,7,9,31,33,43,129$ & $0,0,9,8,30,33,168$ \\
\hline
\hline
(29) & $6$ & $2$ & $7$ & 13 & & $5,5$ & $0,2$ \\
(30) & $6$ & $2$ & $7$ & 15 & & $5$ & $1,1,1$\\
(31) & $6$ & $2$ & $7$ & 19 & & $6,6$ & $0,2/4$ \\
(32) & $6$ & $2$ & $7$ & & & $4,8,10,24,14,56$ & $0,0,6,8,30,24$\\
(33) & $6$ & $2$ & $7$ & & & $5,7,11,15,15$ & $0,2,6,10,5$ \\
(34) & $6$ & $2$ & $7$ & 13 & & $5,5,17,9,25,65$ & $0,0,12,4,15,90$\\
\hline
\hline
(35) & $7$ & $2$ & $8$ & $14$ & & $5,7$ & $0,0,0,0,0,84,133,336$
\end{tabular}
\medskip
\caption{Options for the point counts of $C$ and $C'$ in Theorem~\ref{T:purely geometric bounds} when $d > 2$. We have omitted some data not used in the proof of Theorem~\ref{T:cyclic}.}
\label{table:weil polynomials}
\end{table}

\section{Splitting types and splitting sequences}
\label{sec:splitting}

We now formalize the main strategy used in the arguments.
\begin{definition}
For a place of $F$ which does not ramify in $F'$, the \emph{splitting type} of this place
is the partition of $d$ corresponding to the Frobenius conjugacy class of the place in the symmetric group $S_d$;
in other words, it records the relative degrees of the places of $F'$ lying above the original place.
When describing a splitting type, we write $a^b$ to represent $b$ copies of $a$ in the partition.

We then define the \emph{splitting sequence} of $F'/F$
as the sequence $(s_1, s_2, \dots)$
in which $s_i$ is the multiset of splitting types of degree-$i$ places of $F$.
When describing a splitting sequence, we write $(\times n)$ after a partition to indicate that it occurs with multiplicity $n>1$.
\end{definition}

As a first application of the strategy, we record the following consequence of Lemma~\ref{lem:cycle1 prelim}.
\begin{lemma} \label{lem:cycle1}
There exists at least one degree-$1$ place of $F$ which lifts to a degree-$d$ place of $F'$.
\end{lemma}
\begin{proof}
Suppose the contrary. Each degree-1 place of $F$ which does (resp. does not) ramify in $F'$
lifts either to at least one place of $F'$ of degree strictly less than $d/2$ or to at least one place (resp. at least two places) of degree exactly $d/2$.
However, this contradicts \eqref{eq:cycle1 prelim} save for the exceptional case of Lemma~\ref{lem:cycle1 prelim}. In that case,
\eqref{eq:cycle1 prelim2} tells us that $F'/F$ is everywhere unramified and $F'$ has no places of degree 1 or 3; this leaves no possible splitting types for a degree-1 place other than a single degree-5 place.
\end{proof}

\begin{remark} \label{R:recursive splitting}
Given candidate tuples $\ptcount{C}{n}, \ptcount{C'}{n}$ for some $n$,
identifying the splitting sequences compatible with these values is a combinatorial exercise akin to a pencil-and-paper logic puzzle. While this exercise is pleasant enough in each individual instance, given the number of instances involved it is more reliable to automate this process.
We do this by a recursive procedure: given a candidate for the first $n-1$ terms of the splitting sequence, we iterate over possible $n$-th terms (i.e., $a_{n}(F)$-element multisets of partitions of $d$) to see which ones
give the correct values of $\#C'(\FF_{2^n})$.
This completes all cases of the problem referenced throughout this paper in negligible time.
\end{remark}

Since we are in the special situation where $h_{F'} = h_F$, we can make some additional arguments.
These are not strictly necessary, and indeed are not be used in the uniform application of Remark~\ref{R:recursive splitting}
(Lemma~\ref{L:unified calculation}); however, we will use them in some of the human-readable alternate calculations in order to shorten the arguments.
\begin{remark} \label{R:castelnuovo-severi}
Suppose that $\#J(C)(\FF_2)$ is coprime to $d$. Then the composition
\[
J(C)(\FF_2) \to J(C')(\FF_2) \to J(C)(\FF_2),
\]
in which the first map is pullback of divisors and the second map is pushforward of divisors, is multiplication by $d$ and hence an isomorphism. Since we are requiring $h_{F'} = h_F$, the pullback map is an injection between two finite groups of the same order, and hence also an isomorphism. Consequently, any degree-0 divisor on $C'$ which pushes forward to a principal divisor is itself principal.

This immediately implies that no degree-1 place of $F$ can lift to more than one degree-1 place of $F'$. In a similar vein, if some degree-3 place of $F$ lifts to 4 or more degree-3 places of $F'$, then $C'$ admits a $g^r_3$ with $r \geq 2$, which implies $g' \leq 1$.

We can also use this logic in conjunction with the  Castelnuovo--Severi inequality \cite[Theorem~3.11.3]{stichtenoth}: if $C'$ admits two maps to $\PP^1$ which are incommensurable (in that they do not factor through a common map which is not an isomorphism) of degrees $d_1$ and $d_2$, then $g' \leq (d_1-1)(d_2-1)$. This implies that if there are two different degree-3 places of $F$, each of which lifts to 2 or more degree-3 places of $F'$, then
$g' \leq 4$.
\end{remark}

\section{Subfields of the Galois closure}
\label{sec:subfields}

Let $F''$ be the Galois closure of $F'/F$ and let $C''$ be the associated curve.
Let $G$ be the Galois group of $F''/F$, viewed as a transitive subgroup of $S_d$.
Building on \cite[Lemma~8.2]{part1},
it will be profitable to consider some other subfields of $F''$ and the covers of $C$ corresponding to them. In doing so, we must keep in mind that these fields need not be purely geometric extensions of $F$; see for example Remark~\ref{rmk:constant resolvent}.

\begin{definition} \label{D:A5 char table}
Write $q''$ for $q_{F''}$.
Let $G_0$ be the Galois group of $F''$ over $F\cdot \FF_{q''}$.
Then $G_0$ is a normal subgroup of $G$ and $G/G_0 \cong \Gal(\FF_{q''}/\FF_q)$ is cyclic.

Let $\chi_0, \chi_1, \dots$ be the nontrivial irreducible characters of $G$ over $\QQ$,
numbered so that $\chi_0$ is the restriction of the standard representation of $S_d$.
For each $i$, let $d_i$ be the dimension of each irreducible \emph{complex} representation of $G_0$ appearing in $\chi_i$. Then writing $\Res$ for Weil restriction, we have an isogeny decomposition
\begin{equation} \label{eq:isogeny decomp1}
\Res_{\FF_{q''}/\FF_q} J(C'') \sim J(C) \times \prod_{i\geq 0} B_i^{d_i}
\end{equation}
of abelian varieties; when $C' \to C$ is \'etale, we have $\dim(B_i) = (\dim \chi_i)(g-1)$.
In particular, $B_0$ is the (generalized) Prym variety $A$ of the original cover $C' \to C$.

Following the convention of \cite{part1}, we write $T_{*, q}$ for the $q$-Frobenius trace of $*$ (which could be either a curve or an abelian variety); for every positive integer $n$,
\begin{equation} \label{eq:trace of Prym}
\#C'(\FF_{q^n}) = \#C(\FF_{q^n}) - T_{A,q^n}.
\end{equation}

More generally, if $H$ is a subgroup of $G$ which surjects onto $G/G_0$, then
\begin{equation} \label{eq:isogeny decomp2}
J(C''/H) \sim J(C) \times \prod_{i \geq 0} B_i^{c_i}
\end{equation}
where $c_i$ is the multiplicity of $\chi_i$ in the representation of $G$ induced from the trivial representation of $H$. For every positive integer $n$,
\begin{equation} \label{eq:trace of Prym2}
\#(C''/H)(\FF_{q^n}) = \#C(\FF_{q^n}) - \sum_i c_i T_{B_i, q^n}.
\end{equation}
\end{definition}

\begin{definition}
An important special case of Definition~\ref{D:A5 char table} occurs when $G_0 \not\subseteq A_d$ and $H = G_0 \cap A_d$. In this case, the function field of $C''/H$ is the \emph{quadratic resolvent} of $F'/F$.
The decomposition \eqref{eq:isogeny decomp2} reduces to $J(C''/H) \sim J(C) \times B$ where $B$ is the Prym variety of $C''/H \to C$; in the notation of Definition~\ref{D:A5 char table}, $B$ corresponds to the sign representation of $G_0$.
An unramified splitting type for $C' \to C$ corresponds to the splitting type $1^2$ or $2$ for
the quadratic resolvent
according to whether it is the cycle type of an even or odd permutation (i.e., whether or not the number of parts in the partition has the same parity as $d$).
\end{definition}

\begin{remark} \label{rmk:genus 2 double cover}
If the quadratic extension is an \emph{everywhere unramified} quadratic extension,  then by class field theory,
\begin{equation} \label{eq:double cover cft}
\#J(C)(\FF_2) \equiv 0 \pmod{2}.
\end{equation}
\end{remark}

\begin{remark} \label{rmk:constant resolvent}
We must account for the possibility that the quadratic resolvent is a constant extension rather than a purely geometric one. In this case, the Frobenius conjugacy class of a degree-$i$ place of $F$ must be odd if $i$ is odd and even if $i$ is even. That is, $B$ is the quadratic twist of $J(C)$.
\end{remark}

From Lemma~\ref{lem:cycle1} and Remark~\ref{rmk:constant resolvent}, we immediately deduce the following.
\begin{lemma} \label{lem:cycle}
The group $G$ contains a $d$-cycle. Consequently:
\begin{enumerate}
\item[(a)] if $d$ is even, then $G \not\subseteq A_d$;
\item[(b)] if $d$ is odd, then either $G \subseteq A_d$ or 
the quadratic resolvent is purely geometric.
\end{enumerate}
\end{lemma}

\begin{remark} \label{R:basic paradigm}
We can now describe the basic paradigm that we will use to address the cases $d=3,4,5,6$ of
Theorem~\ref{T:cyclic}. We start with an option for the Weil polynomials of $C$ and $C'$.
We then use the process indicated in Remark~\ref{R:recursive splitting} to compute possible values of the first few terms of the splitting sequence consistent with the point counts of $C$ and $C'$,
any known limitations on $G$ (which in some cases get stricter as we compute more terms), and the constraint that the traces of any abelian variety known to occur as isogeny factors of $J(C'')$
must come from a Weil polynomial of the appropriate degree (see below).
\end{remark}

\begin{remark} \label{R:extreme traces}
The set of Weil polynomials corresponding to abelian varieties of dimension $\leq 6$ over $\FF_2$ can be recovered using \SageMath{} as in \cite{part1}. The results have also been tabulated in LMFDB;
note that to look up $T_{A,q}$ for an abelian variety $A$ in LMFDB, one should look at the ``number of points on the curve'' over $\FF_q$ which computes $q+1-T_{A,q}$.

For the benefit of the human reader, we spell out some of the most relevant constraints on the traces of an abelian variety $A$ over $\FF_2$.
\begin{itemize}
\item
We have  
\begin{equation} \label{eq:trace upper bound}
|T_{A,2}| \leq 2 \dim(A)
\end{equation}
with equality iff $A$ is isogenous to a power of an elliptic curve with trace $\pm 2$
(e.g., by \cite[Theorem~2.1.1]{serre-rational}).
When equality occurs, we have $T_{A,4} = 0$.
\item
For $\dim(A) = 2$,
\begin{align}
\label{eq:dim 2 trace -3}
T_{A,2} = -3 &\Longrightarrow T_{A,4} \in \{-1, -3\}, \\
\label{eq:dim 2 trace -2}
T_{A,2} = -2 &\Longrightarrow T_{A,4} \in \{-6,-4,-2,0\}.
\end{align}
\end{itemize}
\end{remark}
\section{Subfields by degree}

We now make explicit some consequences of the discussion from \S\ref{sec:subfields} for $d=3,4,5,6$,
then describe a unified calculation that addresses most cases.

\begin{remark}
Suppose that $d=3$ and $G_0 = G = S_3$. 
The equation \eqref{eq:trace of Prym2} specializes to
\begin{equation} \label{eq:cubic trace of resolvent}
\#C''(\FF_q) = \#C(\FF_q) - 2 T_{A,q} - T_{B,q}.
\end{equation}
In particular, if $\#C'(\FF_q) = 0$, then $\#C''(\FF_q) = 0$ and combining \eqref{eq:trace of Prym} with \eqref{eq:cubic trace of resolvent} yields 
\begin{equation} \label{eq:cubic trace of resolvent2}
T_{A,q} = -T_{B,q} = \#C(\FF_q).
\end{equation}
\end{remark}

\begin{remark}
Suppose that $d=4$ and $G = S_4$. The maximal constant subextension of $F''/F$ is cyclic, so it must equal either $F$ or the quadratic resolvent. 

Suppose now that $C' \to C$ is \'etale. If the quadric resolvent is constant, then
the cubic resolvent is a cyclic cubic \'etale cover of $C_{\FF_4}$, forcing
\begin{equation} \label{eq:A4 mod 3 condition}
\#J(C)(\FF_4) \equiv 0 \pmod{3}.
\end{equation}
Similarly, if the quadratic resolvent is purely geometric, then it admits a cyclic cubic \'etale cover, forcing
\begin{equation} \label{eq:S4 mod 3 condition}
\#J(C)(\FF_2) \#B(\FF_2) \equiv 0 \pmod{3}.
\end{equation}
\end{remark}

\begin{remark} \label{R:degree 5}
Suppose that $d=5$, $G =A_5$, and $C' \to C$ is \'etale. Since $A_5$ is simple, this implies $G_0 = A_5$.
We may number the characters in Definition~\ref{D:A5 char table} so that $\dim(B_1) = 5(g-1)$, $\dim(B_2) = 6(g-1)$.
Then \eqref{eq:trace of Prym2} specializes to
\begin{align}
\#(C''/D_5)(\FF_{q}) &= \#C(\FF_{q}) - T_{B_1, q} \label{eq:trace of Prym2 A51}\\
\#(C''/A_3)(\FF_{q}) &= \#C(\FF_{q}) - 2 T_{A,q} - T_{B_1, q} - T_{B_2, q}.
\label{eq:trace of Prym2 A52}
\end{align}
The conversion of splitting types from $C'$ to these quotients is as follows:
\begin{center}
\begin{tabular}{c|c|c}
Type in $C'$  & Type in $C''/D_5$ & Type in $C''/A_3$ \\
\hline
$5$ & $5+1$ & $5^4$ \\
$3+1^2$ & $3^2$ & $3^6 + 1^2$ \\
$2^2+1$ & $2^2 + 1^2$ & $2^{10}$ \\
$1^5$ & $1^6$ & $1^{20}$
\end{tabular}
\end{center}
\end{remark}

\begin{remark} \label{rmk:degree 6}
Suppose that $d=6$, $G = S_6$, and $C' \to C$ is \'etale. The maximal constant subextension of $F''/F$ is cyclic, so it must equal either $F$ or the quadratic resolvent. In either case, we may number the characters 
in Definition~\ref{D:A5 char table} so that $\chi_1$
is the image of $\chi_0$ under the action of an outer automorphism of $S_6$; then $\dim(B_1) = 5(g-1)$. 
The quotient by $C''$ by the image of $S_5$ under an outer automorphism is a cover of $C$ whose Jacobian is isogenous to $J(C) \times B_1$; we call this the \emph{sextic twin} of the original cover. Note that even if the quadratic resolvent is constant, the outer automorphism of $S_6$ preserves $A_6$, so the sextic twin descends canonically to $\FF_2$.

For reference, we list here the possible splitting types for unramified places and the effect of an outer automorphism on these types.
\begin{align*}
\mbox{odd}: & \qquad 6 \leftrightarrow {3+2+1}, \qquad {4+1^2}, \qquad {2^3} \leftrightarrow {2+1^4}, \\
\mbox{even}: & \qquad {5+1}, \qquad {4+2}, \qquad {3^2} \leftrightarrow {3+1^3}, \qquad {2^2+1^2}, \qquad {1^6}.
\end{align*}
\end{remark}

\begin{remark} \label{rmk:degree 6 pgl25}
Suppose that $d=6$, $G = \PGL(2, 5)$, and $C' \to C$ is \'etale. The splitting types $3+2+1$, $2+1^4$, $4+2$, and $3+1^3$
cannot occur,
as these correspond via the outer automorphism to splitting types with no singletons.

Following the model of Remark~\ref{rmk:degree 6}, we may construct a sextic twin, but in this case it is reducible: it is the disjoint union of $C$ with a degree-5 cover. The Jacobian of the sextic twin is isogenous to $J(C)^2 \times B_1'$ with $\dim(B_1') = 4(g-1)$.
\end{remark}

Implementing Remark~\ref{R:basic paradigm}, we obtain the following via a unified calculation.
We will also step through the cases individually in subsequent sections; this will help to illustrate
why some of the conditions appear.
\begin{lemma} \label{L:unified calculation}
Consider a candidate pair of point count sequences
\[
\ptcount{C}{\infty}, \qquad \ptcount{C'}{\infty}
\]
from \cite[Theorem~1.3(b)]{part1} for which $3 \leq d \leq 6$ and $g' = dg-d+1$. 
Choose $G \subseteq S_d$ such that
\[
(d,G) \in \{(3, S_3), (4, S_4), (5, A_5), (6, S_6), (6, \PGL(2,5))\}.
\]
If $d$ is odd, set $G_0 := G$; otherwise, choose $G_0 \in \{G, G \cap A_d\}$.
Then there is no splitting sequence of length $7$ compatible with the point counts and all of the following restrictions.
\begin{itemize}
\item
Restrictions on $J(C)$:
\begin{itemize}
\item
If $G_0 \not\subseteq A_d$, then the mod-$2$ congruence \eqref{eq:double cover cft} holds.
\item
If $d=4$ and $G_0 = A_4$, then the mod-$3$ congruence \eqref{eq:A4 mod 3 condition} holds.
\end{itemize}
\item
Restrictions on splitting types:
\begin{itemize}
\item
If $G \neq G_0$, then Remark~\ref{rmk:constant resolvent} applies.
\item
If $(d,G) = (5, A_5)$, then the splitting types ${4+1}, {3+2}, {2+1^3}$ do not occur (Remark~\ref{R:degree 5}).
\item
If $(d,G) = (6, \PGL(2,5))$, then the splitting types ${4+2}, {3+2+1}$, ${3+1^3}, {2+1^4}$ do not occur (Remark~\ref{rmk:degree 6 pgl25}).
\end{itemize}
\item
Compatibility with Weil polynomials:
\begin{itemize}
\item
If $G_0 \not\subseteq A_d$, then the traces of $B$ are compatible with a Weil polynomial.
Moreover, if $g=2$ and $\ptcount{C}{2} \in \{(2,8), (4,8)\}$, then $T_{B,2} \in \{\pm 2\}$ (Remark~\ref{R:genus 2 data2}).
\item
If $d=5$, then the traces of $B_1$ and $B_2$ are compatible with Weil polynomials (Remark~\ref{R:degree 5}).
\item
If $d=6$ and $G_0 = G$, then the traces of $B_1$ (if $G = S_6$) or $B_1'$ (if $G = \PGL(2,5)$)
are compatible with a Weil polynomial (Remark~\ref{rmk:degree 6}, Remark~\ref{rmk:degree 6 pgl25}).
\end{itemize}
\end{itemize}
\end{lemma}

In the remaining sections, we proceed through the values $d>2$ allowed by Theorem~\ref{T:purely geometric bounds}. For ease of reference, when citing Table~\ref{table:weil polynomials} we will specify rows individually or in short ranges using an abbreviated syntax: e.g., ``Table~\ref{table:weil polynomials}(4--6)'' refers to rows (4), (5), and (6) of the table.

\section{Degree 3}

For $d=3$, we rule out $G = S_3$ using an analysis of quadratic extensions.

\begin{lemma} \label{L:degree 3 cyclic}
If $d=3$, then $C' \to C$ is Galois and cyclic.
\end{lemma}
\begin{proof}
Suppose to the contrary that $G = S_3$.
By Theorem~\ref{T:purely geometric bounds}, we have $(g,g') \in \{(2,4), (2, 6), (3,7), (4,10)\}$.
By Lemma~\ref{lem:cycle}, the quadratic resolvent is purely geometric, so $G_0 = S_3$
and the mod-2 congruence \eqref{eq:double cover cft} applies unless $(g,g') = (2,6)$.
Let $C''' := C''/A_3$ be the curve corresponding to the quadratic resolvent.

We first treat the case $(g,g') = (2,6)$.
Table~\ref{table:weil polynomials}(3) shows that 
\[
\#C(\FF_2) \geq 3, \qquad \#C'(\FF_2) = 0, \qquad \#C'(\FF_4) = 2. 
\]
By Riemann--Hurwitz, $C' \to C$ is ramified at either one or two geometric points of $C'$. Since $\#C'(\FF_2) = 0$, $C' \to C$ must ramify at the unique degree-2 place of $F'$; moreover,
this must be a triple point and not a double point (as the latter would force $a_2(F') \geq 2$).
It follows that $C''' \to C$ is \'etale,
and so $\dim(B) = 1$ and $T_{B,2} = -\#C(\FF_2) \leq -3$ by \eqref{eq:cubic trace of resolvent2}, violating \eqref{eq:trace upper bound} 

In the remaining cases, $C' \to C$ is \'etale and we may appeal to the uniform calculation
(Lemma~\ref{L:unified calculation}). 
Alternatively, we may break the individual cases down as follows.
\begin{itemize}
\item
For $(g,g') = (4,10)$, we have $\dim(B) = 3$. From Table~\ref{table:weil polynomials}(12,13),
we have
\[
\#C(\FF_2) \geq 7 \mbox{ or } \ptcount{C}{2} = (6,6), \qquad \#C'(\FF_2) = 0.
\]
By \eqref{eq:cubic trace of resolvent2}, $T_{B,2} \leq -7$ or $T_{B,2} = T_{B,4} = -6$,  contradicting Remark~\ref{R:extreme traces}.

\item
For $(g,g') = (3, 7)$, we have $\dim(B) = 2$.
\begin{itemize}
\item
Table~\ref{table:weil polynomials}(4) is ruled out by the mod-2 congruence \eqref{eq:double cover cft}.
\item
If $\#C'(\FF_2) > 1$, then from Table~\ref{table:weil polynomials}(11),
\[
\#C(\FF_2) = 6, \qquad \ptcount{C'}{3} = (2,2,14).
\]
There are not enough places of $F'$ of degree at most 3 to cover the degree-1 places of $F$.
\item
If $\#C'(\FF_2) = 1$, then some degree-1 place of $F$ has splitting type ${2+1}$, so $\#C'(\FF_4) \geq 3$. 
This rules out Table~\ref{table:weil polynomials}(9); from Table~\ref{table:weil polynomials}(10), 
\[
\ptcount{C}{2} = (5,9), \qquad \ptcount{C'}{2} = (1,3).
\]
The splitting sequence begins
$\{3(\times 4), 2+1\}, \{3(\times 2)\}$.
This means that $\ptcount{C'''}{2} = (8,18)$ and $\tracecount{B}{2} = (-3, -9)$, contradicting \eqref{eq:dim 2 trace -3}.
\item
If $\#C'(\FF_2) = 0$ and $\#C'(\FF_4) > 0$, then from
Table~\ref{table:weil polynomials}(8), 
\[
\ptcount{C}{2} = (4, 12).
\]
By \eqref{eq:cubic trace of resolvent2}, $T_{B,2} = -4$.
By Remark~\ref{R:extreme traces}, $B$ is isogenous to the square of an elliptic curve with trace $-2$,
so the relative quadratic twist of $C'''$ is a degree-2 \'etale cover of $C$ with relative class number 1. However, $\#C(\FF_4) = 12$ which violates Remark~\ref{R:genus 2 data}.

\item
If $\#C'(\FF_2) = \#C'(\FF_4) = 0$, then from
Table~\ref{table:weil polynomials}(5--7),
\[
\ptcount{C}{2} \in \{(2, 8), (3, 7), (4, 2), (4, 8)\}.
\]
By \eqref{eq:cubic trace of resolvent2}, $\tracecount{B}{2} = ({-\#C(\FF_{2^i})})_{i=1}^2$,
contradicting Remark~\ref{R:extreme traces}.
\end{itemize}

\item
For $(g,g') = (2,4)$, we have $\dim(B) = 1$. 
The mod-2 congruence \eqref{eq:double cover cft} rules out Table~\ref{table:weil polynomials}(1).
From Table~\ref{table:weil polynomials}(2), we have 
\[
\#C(\FF_2)  = 3, \qquad \#C'(\FF_2) = 0.
\]
By \eqref{eq:cubic trace of resolvent2}, $T_{B,2} = -3$, contradicting Remark~\ref{rmk:genus 2 double cover}.
 \qedhere
\end{itemize}
\end{proof}

\section{Degree 4}

For $d=4$, we rule out $G = D_4$ using an analysis of quadratic extensions.

\begin{lemma} \label{lem:degree 4 subfield}
If $d=4$, then the Galois group of $F'/F$ cannot equal $D_4$.
\end{lemma}
\begin{proof}
By Theorem~\ref{T:purely geometric bounds}, we have $(g_F,g_{F'}) \in \{(2,5), (2,6), (3,9)\}$. Suppose that $F'/F$ contains an intermediate subfield $E$; then 
$g_E = 2g_F-1$ and both $E/F$ and $F'/E$ are purely geometric quadratic extensions with relative class number 1. Applying Remark~\ref{R:genus 2 data} to these extensions, we rule out the cases
$(g_F, g_{F'}) = (2,6), (3,9)$.
For $(g_F, g_{F'}) = (2,5)$, we see from Remark~\ref{R:genus 2 data} that $C$ has $p$-rank 1, as then does the intermediate curve $C'''$ by the Deuring--Shafarevich formula \cite[(7.2)]{part1}.
In particular, $C_{\overline{\FF}_2}$ and $C'''_{\overline{\FF}_2}$ each admit only one \'etale double cover, which forces $F'/F$ to be cyclic.
\end{proof}

\begin{lemma} \label{lem:degree 4}
If $d=4$, then $C' \to C$ is Galois and cyclic.
\end{lemma}
\begin{proof}
Suppose to the contrary that $G \neq C_4$.
By Theorem~\ref{T:purely geometric bounds}, we have $(g,g') \in \{(2,5), (2, 6), (3,9)\}$.
We have $G \neq D_4$ by Lemma~\ref{lem:degree 4 subfield} and
$G \not\subseteq A_4$ by Lemma~\ref{lem:cycle}, so $G = S_4$.
The maximal constant subextension of $F''/F$ is cyclic, and so must equal either $F$ or the quadratic resolvent.
In the former case, let $C''' := C''/A_4$ be the curve corresponding to the quadratic resolvent and let $F'''$ be its function field.

We first treat the case $(g,g') = (2,6)$.
The map $C' \to C$ ramifies at one geometric point, which must be $\FF_2$-rational.
By Table~\ref{table:weil polynomials}(18), we have $a_2(F') = 0$, so the ramified place of $F'$ must be alone in its fiber, violating Riemann--Hurwitz.

In the remaining cases, $C' \to C$ is \'etale and we may appeal to the uniform calculation
(Lemma~\ref{L:unified calculation}). Alternatively, we may break the individual cases down as follows.
If $(g,g') = (2,5)$ and $\#C'(\FF_2) = 1$, then by Table~\ref{table:weil polynomials}(14), $\#C'(\FF_8) = 1$ and there is no way to accommodate the degree-1 place of $F'$ in a fiber. To handle the remaining cases, suppose first that the quadratic resolvent is constant.
\begin{itemize}
\item
For $(g,g') = (3,9)$, 
the mod-3 congruence \eqref{eq:A4 mod 3 condition} rules out Table~\ref{table:weil polynomials}(20).
From Table~\ref{table:weil polynomials}(19),  we have 
\[
\ptcount{C}{2} = (5,9), \qquad
\ptcount{C'}{4} = (0,0,0,28).
\]
By Remark~\ref{rmk:constant resolvent},
the splitting sequence begins $\{4 (\times 5)\}, \{2^2 (\times 2)\}$,
 but this creates too many degree-4 places of $F'$.

\item
For $(g,g') = (2,5)$, the mod-3 congruence \eqref{eq:A4 mod 3 condition} rules out Table~\ref{table:weil polynomials}(15--17).
\end{itemize}

Suppose next that  the quadratic resolvent is purely geometric.
\begin{itemize}
\item
For $(g,g') = (3,9)$, we have $\dim(B) = 2$. 
From Table~\ref{table:weil polynomials}(19,20)
we have 
\[
\#C(\FF_2) \in \{5,6\}, \qquad \#C'(\FF_2) = \#C'(\FF_4) = 0.
\]
Consequently, each degree-1 place of $F$ lifts to a degree-4 place of $F'$,
and hence to a degree-2 place of $F'''$. Since $0 = \#C'''(\FF_2) = \#C(\FF_2) - T_{B,2}$ by \eqref{eq:trace of Prym2}, we have $T_{B,2} = \#C(\FF_2) \geq 5$, contradicting \eqref{eq:trace upper bound}.

\item
For $(g,g') = (2,5)$, we have $\dim(B) = 1$. From Table~\ref{table:weil polynomials}(15--17) we have 
\begin{equation} \label{eq:gg 25}
\ptcount{C}{2} \in \{(2,8), (4,4), (4,8)\}
\end{equation}
and $\#J(C)(\FF_2) \not\equiv 0 \pmod{3}$. 
By the mod-3 congruence \eqref{eq:S4 mod 3 condition}, $\#B(\FF_2) \equiv 0 \pmod{3}$;
by \eqref{eq:trace upper bound} this forces $T_{B,2} = 0$.
By this plus Remark~\ref{R:genus 2 data2} and \eqref{eq:gg 25}, we must have $\ptcount{C}{2} = (4,4)$;
from Table~\ref{table:weil polynomials}(16), we have $\ptcount{C'}{2} = (0,4)$.
Now $\#C'''(\FF_2) = \#C(\FF_2)$ because $T_{B,2} = 0$, so two of the four degree-1 places of $F$ have even splitting types in
$F'$. Since $\#C'(\FF_2) = 0$, the splitting type $2^2$ must occur twice, but this forces the contradiction $\#C'(\FF_4) \geq 8$.
\qedhere
\end{itemize}
\end{proof}

\section{Degree 5}

For $d=5$, we rule out $G \in \{D_5, S_5\}$ using an analysis of quadratic extensions. 
\begin{lemma} \label{lem:degree 5 subfield}
If $d=5$, then the Galois group of $F'/F$ cannot equal $D_5$ or $S_5$.
\end{lemma}
\begin{proof}
By Theorem~\ref{T:purely geometric bounds}, $g=2$ and $C' \to C$ is \'etale.
By Lemma~\ref{lem:cycle}, $G$ contains a 5-cycle, so if $G = S_5$ then the quadratic resolvent $C''' := C''/A_5$ is purely geometric, while if $G = D_5$ then $C''' := C''/C_5$ is a purely geometric degree-2 \'etale cover of $C$. In both cases, the mod-2 congruence \eqref{eq:double cover cft} holds,
which rules out Table~\ref{table:weil polynomials}(22--25,27).

At this point, by Table~\ref{table:weil polynomials}(21,26,28),
$\#C(\FF_2) \geq 3$ and every degree-1 point of $C$ lifts to a degree-5 point of $C'$.
Let $B'$ be the Prym of $C''' \to C'$, so that $\dim(B') = 1$.
If $G = S_5$, then $\#C'''(\FF_2) = 0$ and so by \eqref{eq:trace of Prym2},
$T_{B',2} = \#C(\FF_2) - \#C'''(\FF_2) \geq 3$; if $G = D_5$, then $\#C'''(\FF_2) = 2 \#C(\FF_2)$ and so by \eqref{eq:trace of Prym2},
$T_{B',2} = \#C(\FF_2) - \#C'''(\FF_2) \leq -3$. In both cases, we violate \eqref{eq:trace upper bound}.
\end{proof}

\begin{lemma} \label{L:degree 5 cyclic}
If $d=5$, then $C' \to C$ is Galois and cyclic.
\end{lemma}
\begin{proof}
Suppose by way of contradiction that $G \neq C_5$.
By Theorem~\ref{T:purely geometric bounds}, $g=2$ and $C' \to C$ is \'etale;
by Lemma~\ref{lem:degree 5 subfield}, $G = A_5$.
We may thus appeal to the uniform calculation
(Lemma~\ref{L:unified calculation});
alternatively, we may break down Table~\ref{table:weil polynomials}(21--28) as follows.
We start with some rows for which computing the splitting sequence already yields a contradiction (even if the cover is cyclic).

\begin{itemize}
\item
From Table~\ref{table:weil polynomials}(21), we have
\[
\#C(\FF_2) = 4, \qquad \ptcount{C'}{5} = (0,6,0,18,0).
\]
There are not enough places of $F'$ of degree at most 5 to cover the degree-1 places of $F$.

\item
From Table~\ref{table:weil polynomials}(22), we have
\[
\ptcount{C}{3} = (4,10,7), \qquad \ptcount{C'}{3} = (0,2,15).
\]
The splitting sequence begins 
$\{5 (\times 3)\}, \{5 (\times 2), 2^2+1\}, \{1^5\}$.
Since $\#J(C)(\FF_2) = 11$ is coprime to $5$, 
the degree-3 places create a contradiction as per Remark~\ref{R:castelnuovo-severi}.

\item
From Table~\ref{table:weil polynomials}(23), we have
\[
\ptcount{C}{2} = (6,6), \qquad \ptcount{C'}{2} = (1,3).
\]
There is no way to include the degree-2 place of $F'$ in a fiber.

\item
From Table~\ref{table:weil polynomials}(24), we have
\[
\ptcount{C}{3} = (5,9,5), \qquad \ptcount{C'}{3} = (0,6,3).
\]
There is no way to include the degree-3 place of $F'$ in a fiber.

\end{itemize}

For the remaining rows of Table~\ref{table:weil polynomials}, let $B_1, B_2$ be the abelian varieties
described in Remark~\ref{R:degree 5}, so that $\dim(B_1) = 5$, $\dim(B_2) = 6$.
We combine the analysis of splitting sequences with the point counts of $B_1$ and/or $B_2$ from \eqref{eq:trace of Prym}, \eqref{eq:trace of Prym2 A51}, and \eqref{eq:trace of Prym2 A52},
then compare to Weil polynomial data (see Remark~\ref{R:Weil poly dim 6}).

\begin{itemize}
\item
From Table~\ref{table:weil polynomials}(25), we have
\begin{gather*}
\ptcount{C}{5} = (3, 5, 9, 33, 33),  \qquad
\ptcount{C'}{5} = (0, 0, 0, 20, 15).
\end{gather*}
The splitting sequence begins $\{5 (\times 3)\}, \{5\}, \{5 (\times 2)\}, ?, \{5 (\times 6)\}$ and
\[
\tracecount{B_2}{5} = (-3, -5, -9, ?, -48),
\]
but the latter is inconsistent with Remark~\ref{R:Weil poly dim 6}.

\item
From Table~\ref{table:weil polynomials}(26), we have
\[
\ptcount{C}{3} = (4, 8, 10), \qquad \ptcount{C'}{3} = (0, 0, 15).
\]
The splitting sequence begins $\{5 (\times 4)\}, \{5 (\times 2)\}, \{5, 1^5\}$ and
\[
\tracecount{B_2}{3} = (-4, -8, -25),
\]
but the latter is inconsistent with Remark~\ref{R:Weil poly dim 6}.

\item
From Table~\ref{table:weil polynomials}(27), we have
\[
\ptcount{C}{3} = (5, 9, 5), \qquad \ptcount{C'}{3} = (0, 10, 0).
\]
The splitting sequence begins $\{5 (\times 5)\}, \{5, 1^5\}, \emptyset$ and
\[
\tracecount{B_2}{3} = (-5, -19, -5),
\]
but the latter is inconsistent with Remark~\ref{R:Weil poly dim 6}.

\item
From Table~\ref{table:weil polynomials}(28), we have
\begin{gather*}
\ptcount{C}{7} = (3,7,9,31,33,43,129), \,
\ptcount{C'}{7} = (0,0,9,8,30,33,168).
\end{gather*}
The splitting sequence begins  
\begin{gather*}
\{5 (\times 3)\},\qquad
\{5 (\times 2)\}, \qquad \{{3+1^2}, {2^2+1}\}, \\
\{5 (\times 5), {3+1^2}\} \mbox{ or } \{5 (\times 4),2^2+1 (\times 2)\}, \\
\{5 (\times 4), 3+1^2, 2^2+1\} \mbox{ or } \{5 (\times 3), 2^2+1 (\times 3)\},  \\
\{5 (\times 4), 3+1^2\} \mbox{ or } \{5 (\times 3), {2^2+1} (\times 2)\}
\end{gather*}
and 
$\tracecount{B_2}{6} = (-3,-7,3,{-27}/{-7}, {-28}/{-3}, {-49}/{-19})$. 
By comparing with Remark~\ref{R:Weil poly dim 6}, we deduce that
\[
\tracecount{B_2}{7} = (-3,-7,3,-7, -3, -19,25);
\]
the splitting sequence continues
\begin{gather*}
\{5 (\times 4), {3+1^2} (\times 2), {2^2+1} (\times 10), 1^5 (\times 2)\} \\
\mbox{ or }
\{5 (\times 4), {3+1^2} (\times 6), {2^2+1} (\times 7), 1^5\} 
\mbox{ or } 
\{5 (\times 4), {3+1^2} (\times 10), {2^2+1} (\times 4)\}
\end{gather*}
and we obtain $\tracecount{B_1}{7} = (0,0,0,-8,-30,-24,{-126}/{-42}/42)$. However, the Weil polynomial of $B_1$ is uniquely determined by $\tracecount{B_1}{5}$ and forces $T_{B_1,2^7} = 0$.
\qedhere
\end{itemize}

\end{proof}

\begin{remark} \label{R:Weil poly dim 6}
We summarize here the Weil polynomial data used in the proof of Lemma~\ref{L:degree 5 cyclic}. For $\dim(A) = 6$,
\begin{align*}
\tracecount{A}{3} = (-3,-5,-9) &\Longrightarrow T_{A,2^5} \geq -13 \\
\tracecount{A}{2} = (-4,-8) &\Longrightarrow T_{A,2^3} \geq -16 \\
\tracecount{A}{2} = (-5,-19) &\Longrightarrow
T_{A,2^3} = 25 \\
\tracecount{A}{4} = (-3,-7,3,-27) &\Longrightarrow T_{A,2^5} \geq 17 \\
\tracecount{A}{5} = (-3,-7,3,-7,-28) &\Longrightarrow T_{A,2^6} =41 \\
\tracecount{A}{5} = (-3,-7,3,-7,-3) &\Longrightarrow T_{A,2^6} \geq -19. 
\end{align*}
\end{remark}

\section{Degree 6}

For $d=6$, we have a number of Galois groups to worry about.
We deal with most of these by analyzing intermediate subfields.

\begin{lemma} \label{lem:degree 6 subfield}
If $d=6$, then $F'/F$ does not contain an intermediate subfield.
\end{lemma}
\begin{proof}
Suppose to the contrary that $E$ is an intermediate subfield;
then both $E/F$ and $F'/E$ are extensions with relative class number 1.
By Theorem~\ref{T:purely geometric bounds}, we have $(g_F, g_{F'}) = (2,7)$.

In case $[E:F] = 3$, from Table~\ref{table:weil polynomials}(1,2), the intermediate curve has at most one
$\FF_2$-rational point. Remark~\ref{R:genus 2 data} then shows that no extension $F'/E$ can exist.

In case $[E:F] = 2$, 
from Table~\ref{table:weil polynomials}(29--34) we have $\#C(\FF_2) \geq 4$.
Comparing with Remark~\ref{R:genus 2 data}, we deduce that the intermediate curve has 2 $\FF_2$-rational points. From Table~\ref{table:weil polynomials}(4--11), we see that the cover $F'/E$ corresponds
to Table~\ref{table:weil polynomials}(5), so the Jacobian of the intermediate curve has isogeny class \avlink{3.2.ab\_c\_a}
which does not appear in 
\cite[Table~5]{part1}. Hence by \cite[Theorem~1.3]{part1}, $F'/E$ cannot be cyclic, contradicting Lemma~\ref{L:degree 3 cyclic}.
\end{proof}

\begin{lemma} \label{lem:degree 6 cyclic}
If $d=6$, then $C' \to C$ is Galois and cyclic.
\end{lemma}
\begin{proof}
Suppose to the contrary that $G \neq C_6$. 
By Theorem~\ref{T:purely geometric bounds}, $g=2$ and $C' \to C$ is \'etale.
By Lemma~\ref{lem:cycle}, $G$ contains a 6-cycle, so $G \not\subseteq A_6$.
By Lemma~\ref{lem:degree 6 subfield} and Lemma~\ref{L:degree 3 cyclic}, $F'/F$ has no intermediate subfields; consequently, $G$ must be either $\PGL(2,5) \cong S_5$ or $S_6$.
We may thus appeal to the uniform calculation
(Lemma~\ref{L:unified calculation}); alternatively, we may break the individual
cases down as follows.

Suppose first that the quadratic resolvent is constant.
Among Table~\ref{table:weil polynomials}(29--34),
for some rows we obtain a contradiction directly from the splitting sequence
by keeping in mind Remark~\ref{rmk:constant resolvent}.
\begin{itemize}
\item
From Table~\ref{table:weil polynomials}(29), we have
\[
\ptcount{C}{2} = (5,5), \qquad \ptcount{C'}{2} = (0,2).
\]
The degree-2 place of $F'$ maps to a degree-1 place of $F$, forcing the splitting type $4+2$ in degree 1.

\item
From Table~\ref{table:weil polynomials}(30), we have
\[
\#C(\FF_2) = 5, \qquad \ptcount{C'}{3} = (1, 1, 1).
\]
The degree-1 place of $F'$ forces the splitting type $5+1$ in degree 1.

\item
From Table~\ref{table:weil polynomials}(31), we have
\[
\ptcount{C}{2} = (6,6), \qquad \ptcount{C'}{2} = (0,2/4).
\]
The degree-2 places of $F'$ must map to \emph{distinct} degree-1 places of $F$,
forcing the splitting type $4+2$ in degree 1.

\item
From Table~\ref{table:weil polynomials}(32), we have
\[
\ptcount{C}{6} = (4, 8,10,24,14,56), \qquad
\ptcount{C'}{6} = (0, 0, 6, 8, 30,24).
\]
There are not enough degree-6 places of $F'$ to cover the degree-1 places of $F$,
forcing the splitting type $3+3$ in degree 1.

\end{itemize}

For the remaining rows (still assuming constant quadratic resolvent), the splitting sequence begins $\{6 (\times 5)\}$ and the splitting type $3+2+1$ is forced to occur in degree 5 for parity reasons, yielding 
$G = S_6$. We now argue in terms of the sextic twin  (Remark~\ref{rmk:degree 6}), keeping in mind that $\dim(B_1) = 5$.
\begin{itemize}
\item
From Table~\ref{table:weil polynomials}(33), we have
\[
\ptcount{C}{5} = (5,7,11, 15,15), \qquad
\ptcount{C'}{5} = (0, 2, 6, 10,5).
\]
The splitting sequence continues 
\begin{gather*}
\{5+1\}, \{6, 4+1^2\} \mbox{ or } \{4+1^2, 2^3\} \mbox{ or } \{{3+2+1} (\times 2)\}, \\
\{5+1 (\times 2)\} \mbox{ or } \{4+2, 2^2+1^2\}  \mbox{ or } \{3^2, 2^2+1^2\};
\end{gather*}
this yields $\tracecount{B_1}{4} = (0,-10,-27/-18/-9, -22/-10)$, 
contradicting Remark~\ref{R:Weil poly dim 45}.

\item
From Table~\ref{table:weil polynomials}(34), we have
\[
\ptcount{C}{6} = (5,5,17,9,25,65),
\qquad
\ptcount{C'}{6} = (0,0,12,4,15,90).
\]
Since $\#J(C)(\FF_2) = 13$ is coprime to $6$, Remark~\ref{R:castelnuovo-severi} implies that
in degree 3 we cannot use the splitting type $2+1^4$ at all, or the splitting type $4+1^2$ more than once.
The splitting sequence thus continues
\begin{gather*}
\emptyset, \{6, {4+1^2}, {3+2+1} (\times 2)\} \mbox{ or }\\
 \{{4+1^2}, {3+2+1} (\times 2), 2^3\} \mbox{ or } \{{3+2+1} (\times 4)\},
\{5+1\}
\end{gather*}
and 
\[
\tracecount{B_1}{4} = (0,-10,{-3}/{-12}/{-21},-10).
\]
Combining this with Remark~\ref{R:Weil poly dim 45} yields
\[
\tracecount{B_1}{6} = (0,-10,-3,-10,10/15,23).
\]
The splitting sequence now begins
\[
\{6 (\times 5)\}, \emptyset, \{{3+2+1} (\times 4)\}, \{5+1\}, \{6, {3+2+1}(\times 3)\},
\]
but there is no possible extension matching $T_{B_1,2^6}$.

\end{itemize}

Suppose next that the quadratic resolvent is purely geometric.
The mod-2 congruence \eqref{eq:double cover cft} rules out Table~\ref{table:weil polynomials}(29--31,34). For Table~\ref{table:weil polynomials}(32,33) we must allow both $G = \PGL(2,5)$ and $G = S_6$; we treat these cases in parallel, listing the numerical values that arise in Table~\ref{table:S6 cases}.
\begin{table}
\Small
\begin{tabular}{c|c|c} 
$(\#C(\FF_{2^i}))_{i=1}^5$ & (4,8,10,24,14) & (5,7,11,15,15) \\
$(\#C'(\FF_{2^i}))_{i=1}^5$ & (0,0,6,8,30) & (0,2,6,10,5) \\
Degree 1 splittings & $\{6 (\times 3), 3^2\}$ & $\{6 (\times 3), 4+2, 3^2\}$ \\
$\tracecount{B}{5}$ & $(2,0,-4,-8,-8)$ & $(1,-3,-5,1,11)$ \\
Degree 2 splittings & $\{(6 / 2^3) (\times 2)\}$ & $\{6 / 2^3\}$ \\
Degree 3 splittings & $\{(4+2 / 3^2) (\times 2)\}$ & $\{(4+2/3^2) (\times 2)\}$ \\
\hline
$\tracecount{B_1}{3}$ (splittings) & $(-2, {-20}/{-14}/{-8}, {-26}/{-17}/{-8})$ & 
$({-1}, {-15}/{-9}, {-25}/{-16}/{-7})$ \\
$\tracecount{B_1}{4}$ (Weil polys) & $({-2}, {-8}, {-8}, {-4}/\cdots/16)$ & $({-1},{-9},{-7},{-13}/\cdots/{19})$ \\
Degree 4 splittings & $\{{4+2} (\times 2), {3+2+1} (\times 2)\}$ & $\{6, {3+2+1}\}$ \\
$T_{B_1,2^4}$ from splittings & $0$ & ${-13}$ \\
$T_{B_1,2^5}$ from Weil polys & $33$ & 49 \\
\hline
$\tracecount{B_1'}{3}$ (splittings) & $(2, {-12}/{-6}/0, {-16})$ & 
$(4, {-8}/{-2}, {-14})$ \\
$\tracecount{B_1'}{4}$ (Weil polys) & $(2, 0, {-16},{-8}/{-4})$ & 
none \\
Degree 4 splittings & $\{6, {5+1} (\times 2), 2^3\}/ \{6 (\times 2), {3^2}, {2^2+1^2}\}$ & none \\
$T_{B_1',2^4}$ from splittings & ${-4}$ & none \\ 
$T_{B_1',2^5}$ from Weil polys & 2 & none \\
\end{tabular}
\caption{Numerics from the proof of Lemma~\ref{lem:degree 6 cyclic}.}
\label{table:S6 cases}
\end{table}

\begin{itemize}
\item
We first use the facts that $\dim(B) = 1$ and $|T_{B,2}| \leq 2$
to determine the first term of the splitting sequence.

\item
Let $C''' := C''/(G \cap A_6)$ be the curve corresponding to the quadratic resolvent. 
We then compute $T_{B,4}, T_{B,8}$ and then $\#C'''(\FF_4), \#C'''(\FF_8)$. These values imply
that the splitting types must all be odd in degree 2 and even in degree 3.

\item
We then compute candidates for the second and third terms of the splitting sequence compatible with $B$, keeping in mind the splitting types that cannot be used
if $G  = \PGL(2,5)$;
derive from these the possible values of $\tracecount{B_1}{3}$
or $\tracecount{B'_1}{3}$;
and identify Weil polynomials for $B_1$ or $B_1'$ matching these values.
In each case, we either reach a contradiction or find a unique candidate for the first three terms of the splitting sequence (although not yet a unique Weil polynomial).

\item
We then repeat the previous step for the fourth term of the splitting sequence.
In each remaining case, we determine at most one possible value for $T_{B_1,2^4}$ or $T_{B_1',2^4}$.

\item
We now find that there are no options for the fifth term of the splitting sequence 
consistent with both the point counts and the Weil polynomial constraints.
\qedhere
\end{itemize}

\end{proof}

\begin{remark} \label{R:Weil poly dim 45}
We summarize here the Weil polynomial data used in the proof of Lemma~\ref{lem:degree 6 cyclic}. For $\dim(A) = 4$,
\begin{align*}
\tracecount{A}{2} = (2,-12) &\Longrightarrow T_{A,2^3} \geq -10 \\
\tracecount{A}{2} = (2,-6) &\Longrightarrow T_{A,2^3} \geq -10 \\
\tracecount{A}{3} = (2,0,-16) &\Longrightarrow T_{A,2^4} \in \{-8,-4\} \\
\tracecount{A}{2} = (4,-8) &\Longrightarrow T_{A,2^3} \geq -11 \\
\tracecount{A}{2} = (4,-2) &\Longrightarrow T_{A,2^3} \geq -8.
\end{align*}
For $\dim(A) = 5$,
\begin{align*}
\tracecount{A}{3} = (0,-10) &\Longrightarrow T_{A,2^3} \geq -18 \\
\tracecount{A}{3} = (0,-10,-18) &\Longrightarrow T_{A,2^4} = 18 \\
\tracecount{A}{3} = (0,-10,-9) &\Longrightarrow T_{A,2^4} \geq -6 \\
\tracecount{A}{3} = (0,-10,-12) &\Longrightarrow T_{A,2^4} \geq 2 \\
\tracecount{A}{4} = (0,-10,-3,-10) &\Longrightarrow T_{A,2^5} \in \{10,15\}, T_{A,2^6} = 23 \\
T_{A,2} = -2 &\Longrightarrow T_{A,2^2} \geq -18 \\
\tracecount{A}{2} = (-2,-14) &\Longrightarrow T_{A,2^3} \geq -2 \\
\tracecount{A}{2} = (-2,-8) &\Longrightarrow T_{A,2^3} \geq -11 \\
\tracecount{A}{3} = (-2,-8,-8) &\Longrightarrow -4 \leq T_{A,2^4} \leq 16 \\
\tracecount{A}{2} = (-1, -15) &\Longrightarrow T_{A,2^3} \geq -1 \\
\tracecount{A}{2} = (-1, -9) &\Longrightarrow T_{A,2^3} \geq -13 \\
\tracecount{A}{3} = (-1, -9, -7) &\Longrightarrow -13 \leq T_{A,2^4} \leq 19.
\end{align*}
\end{remark}

\begin{remark}
The astute reader will have noticed that Lemma~\ref{lem:degree 6 subfield} also rules out $G = C_6$, so Lemma~\ref{lem:degree 6 cyclic} can be upgraded to say that the case $d=6$ cannot occur at all! This is consistent with \cite[Table~5]{part1}. By contrast, the proof of Lemma~\ref{lem:degree 4 subfield} must work around the fact that $d=4$ does occur in \cite[Table~5]{part1}.
\end{remark}

\section{Degree 7}

For $d=7$, the complexity of the representation theory of $S_7$ and the presence of additional transitive subgroups (notably $\PSL(2,7)$) together make it quite difficult to follow the paradigm of Remark~\ref{R:basic paradigm}. Instead, we make a detailed analysis of polarizations in the style of \cite{howe2021}.

\begin{lemma} \label{lem:degree 7 cyclic}
If $d=7$, then $C' \to C$ is Galois and cyclic.
\end{lemma}
\begin{proof}
By Theorem~\ref{T:purely geometric bounds}, we have $g=2$.
From Table~\ref{table:weil polynomials}, we may read off the Weil polynomials of $C$ and $A$:
they are $P_1$ and $P_1P_2$ for
\[
P_1(T) = T^2 + 2T - 1, \qquad P_2(T) = T^6 - 5T^5 - 3T^4 + 43T^3 - 33T^2 - 59T + 43.
\]
We analyze this case following \cite[Theorem~4.5]{howe2021}.
The polynomials $P_1, P_2$ are both irreducible; let $K_1 = \QQ[\pi_1]/(P_1(\pi_1)), K_2 = \QQ[\pi_2]/(P_2(\pi_2))$ be the number fields defined by $P_1, P_2$.
Using \Magma, we compute that $\ZZ[\pi_1, \overline{\pi}_1], \ZZ[\pi_2, \overline{\pi}_2]$ are the maximal orders of $K_1, K_2$, and both orders have class number 1.
It follows that up to isomorphism there are unique abelian varieties $A_1, A_2$ over $\FF_2$ with respective Weil polynomials $P_1, P_2$;
in particular, $A_1 \cong J(C)$.

By \cite[Lemma~9.3]{part1}, there is an exact sequence
\[
0 \to \Delta \to J(C) \times_{\FF_2} A_2 \to J(C') \to 0
\]
in which $\Delta$ is a nontrivial finite flat group scheme killed by 7.
In $K_2$, the rational prime $7$ decomposes as $\frakp^6$ where
$\frakp = (\zeta_7 - 1)$ for some nontrivial seventh root of unity
$\zeta_7 \in K_2$. Consequently, $\Delta$ must be isomorphic to the kernel of $\zeta_7-1$ acting on $A_2$. 

We now compare with the proof of \cite[Proposition~4.2]{howe2021}.
The principal polarization on $J(C')$ pulls back to a polarization $(\lambda_1, \lambda_2)$ 
on $J(C) \times_{\FF_2} A_2$ of degree $(\# \Delta)^2$. 
The automorphism $1 \times \zeta_7$ acts on $J(C) \times_{\FF_2} A_2$ with its polarization
and acts trivially on the image of $\Delta$, so $J(C')$ admits an automorphism of order 7
that is compatible with its polarization and equivariant with respect to $J(C) \to J(C')$.
By Torelli \cite[Theorem~12.1]{milne}, $C'$ admits an automorphism of order 7 over $C$,
proving the claim.
\end{proof}

As promised, we adapt the argument to cover a related statement that was asserted without proof in \cite[Remark~6.2]{part1}.

\begin{lemma} \label{lem:genus 6 cover}
Let $C_0$ be the curve $y^2 + y = x^5$ over $\FF_2$.
Let $C$ be a curve of genus $6$ over $\FF_2$ such that $\ptcount{C}{6} = (0,0,0,20,15,90)$.
Then there exists a cyclic \'etale morphism $C \to C_0$ of degree $5$. In particular, $C$ is unique up to isomorphism.
\end{lemma}
\begin{proof}
The real Weil polynomial of $C$ factors as
\[
(T-2)(T+2)(T^4 - 3T^3 - 6T^2 + 18T + 1).
\]
Let $E_1,E_2$ be elliptic curves over $\FF_2$ with $\#E_1(\FF_2) = 1, \#E_2(\FF_2) = 5$.
By \cite[Theorem~4.2]{howe2021}, there exist morphisms $f_1\colon C \to E_1$, $f_2\colon C \to E_2$ of degrees dividing $2^2 \times 5$, $2^2 \times 19$ respectively. 
By \cite[Theorem~4.2]{howe2021} again, there exist morphisms $g_1\colon C_0 \to E_1$, $g_2\colon C_0 \to E_2$ of degrees dividing 4.
The composition $f_{i*} \circ f^*_i$ equals the isogeny $[\deg(f_i)]$ on $J(E_i)$;
similarly, the composition $g_{i*} \circ g^*_i$ equals the isogeny $[\deg(g_i)]$ on $J(E_i)$.

The composition of $g_1^* \times g_2^*: E_1 \times E_2 \to J(C_0)$ followed by $g_{1*} \times g_{2*}: J(C_0) \to E_1 \times E_2$ is multiplication by a power of 2 on $E_1 \times E_2$, which in particular is an isogeny. Since $E_1 \times E_2$ and $J(C_0)$ are of the same dimension, it follows that both maps in the composition are also isogenies, and their composition in the opposite order is also multiplication by a power of 2 (the same one in fact).

Let $f^*\colon J(C_0) \to J(C)$ and $f_*\colon J(C) \to J(C_0)$ be the respective compositions
\[
J(C_0) \stackrel{g_{1*} \times g_{2*}}{\longrightarrow} E_1 \times E_2 \stackrel{f_1^* \times f_2^*}{\longrightarrow} J(C),
\qquad
J(C) \stackrel{f_{1*} \times f_{2*}}{\longrightarrow} E_1 \times E_2 \stackrel{g_1^* \times g_2^*}{\longrightarrow} J(C).
\]
The composition $f_*\circ f^*\colon J(C_0) \to J(C_0)$ is multiplication by an integer 
of the form $2^n 5^i 19^j$ with $n \geq 0$ and $i,j \in \{0,1\}$.
Let $A$ be the reduced closed subscheme of the identity component of $\ker(f_*)$. Then as in \cite[Lemma~9.3]{part1}, we have an exact sequence
\[
0 \to \Delta \to J(C_0) \times_{\FF_2} A \to J(C) \to 0
\]
in which the map $J(C_0) \to J(C)$ is $f^*$ and $\Delta$ is a finite flat group scheme over $\FF_2$ killed by $2^n \times 5 \times 19$ for some $n \geq 0$.
In fact we must have $n = 0$ because $J(C_0)$ is supersingular (so any 2-power-torsion group subscheme of it is biconnected) whereas $A_2$ is ordinary (so any 2-power-torsion group subscheme of it has trivial biconnected constituent). 

Let $P$ be the Weil polynomial of $A$ and put $K_2 = \QQ[\pi]/(P(\pi))$.
Using \Magma, we compute that $\ZZ[\pi, \overline{\pi}]$ is the maximal order of $K$ and its class number is 1. Consequently, $A$ is the unique abelian variety with its Weil polynomial.
Since $\zeta_5 \in K$, we deduce that $A$ admits an automorphism $\alpha$ of order 5.

Since $19 \not\equiv 1 \pmod{5}$, $\alpha-1$ kills the 19-part of $\Delta$.
In the field $K$, the rational prime $5$ decomposes as $\frakp_1^4 \frakp_2^4$
and $\zeta_5-1$ has positive valuation with respect to both $\frakp_1$ and $\frakp_2$; hence $\alpha-1$ also kills the 5-part of $\Delta$.
We conclude that $\alpha-1$ kills all of $\Delta$;
as in the proof of \cite[Proposition~4.2]{howe2021}, we conclude that 
$J(C)$ admits an automorphism of order 5 preserving its principal polarization
and fixing $J(C_0)$.
By Torelli, $C$ admits an automorphism of order 5, the quotient by which is a curve whose Jacobian is isogenous to $J(C_0)$.
From \href{https://www.lmfdb.org/Variety/Abelian/Fq/2/2/a_a}{LMFDB} we see that $J(C_0)$ is the unique Jacobian in its isogeny class, so $C$ is in fact a cyclic degree-5 cover of $C_0$, which by Riemann--Hurwitz must be \'etale. (Note that \emph{a posteriori} the 19-part of $\Delta$ is trivial.)
The map $C \to C_0$ defines an extension of function fields of relative class number 1;
by the uniqueness aspect of \cite[Theorem~1.3(c)]{part1}, this leaves only one possible isomorphism class for $C$.
\end{proof}

\begin{remark}
Note that the proof of Lemma~\ref{lem:degree 7 cyclic} does not rule out the existence of a curve
with the point counts ascribed to $C'$, but which does not cover $C$.
By contrast, the proof of Lemma~\ref{lem:genus 6 cover} does rule out the existence of a curve
with the point counts ascribed to $C$, but which does not cover $C_0$; this situation is more typical of applications of this strategy as described in \cite{howe2021}.
\end{remark}

\begin{remark}
It is now possible to give an alternate brute-force proof of the final assertion of Lemma~\ref{lem:degree 7 cyclic}.
To wit, the isogeny class of $J(C)$ is \avlink{6.2.ad\_c\_a\_f\_am\_q},
whose LMFDB entry includes data from a census of curves of genus 6 over $\FF_2$
(which was unavailable at the time of writing of \cite{part1});
this data shows that $C$ is unique.
\end{remark}

\end{document}